\documentclass[12pt, a4paper]{amsart}

\usepackage[english]{babel}
\usepackage{amssymb,amsmath,amsthm}
\usepackage{verbatim}

\usepackage[all]{xy}
\usepackage{xcolor}

\numberwithin{equation}{section}

\newtheorem{dummy}{dummy}[section]

\newtheorem{definition}[dummy]{Definition}
\newtheorem{theorem}[dummy]{Theorem}
\newtheorem{corollary}[dummy]{Corollary}
\newtheorem{lemma}[dummy]{Lemma}

\newtheorem{remark}[dummy]{Remark}
\newtheorem{example}[dummy]{Example}

\def\C{\mathbb C}
\def\P{\mathbb P}

\def\Q{\mathbb Q}

\def\Z{\mathbb Z}

\def\KK{\mathcal K}

\def\NN{\mathcal N}
\def\OO{\mathcal O}

\def\pt{\mathrm{pt}}

\def\={\;=\;}
\def\bal{\begin{aligned}}
\def\eal{\end{aligned}}
\def\be{\begin{equation}\label}
\def\ee{\end{equation}}

\def\wt{\widetilde}
\def\ol{\overline}
\def\eps{\epsilon}

\DeclareMathOperator{\rk}{rk}
\DeclareMathOperator{\Stab}{Stab}

\def\Hom {\operatorname{Hom}\nolimits}

\title[Bogomolov-Prokhorov invariant as equivariant cohomology]
{The Bogomolov-Prokhorov invariant of surfaces as equivariant cohomology}
\author{Evgeny Shinder}
\begin{document}
\maketitle

\begin{abstract}
For a complex smooth projective surface $M$ with an action of a finite cyclic group $G$ we give a uniform proof of the isomorphism 
between the invariant $H^1(G, H^2(M, \Z))$ and the first cohomology of the divisors fixed by the action, using $G$-equivariant
cohomology. This generalizes the main result of Bogomolov and Prokhorov \cite{BP}. 
\end{abstract}

\section{Introduction}

Let $M$ be a smooth projective complex variety 
and $G$ be a finite group acting faithfully on $M$.
In this note we consider the cohomology group
\be{H1}
H^1(G, H^2(M, \Z)).
\ee
of the singular cohomology $H^2(M,\Z)$ considered as a $G$-module.
It is easy to see that this cohomology group is a stable
birational invariant of the pair $(G,M)$ (see Section \ref{sec-invar}).

\medskip

In a slightly different context, when $M$ is defined over
an arbitrary field $K$, the analogous group
\[
H^1(Gal(K^{sep}/K), Pic(M)) 
\]
has been considered by Manin who used it to study
rationality of del Pezzo surfaces over $K$ \cite{M1, M2}.

\medskip

Bogomolov and Prokhorov used the cohomology group 
\be{H1-BP}
H^1(G, Pic(M)) 
\ee
in \cite{BP} when $M$ is a smooth rational surface and $G$ is a finite group to study conjugacy classes in
the Cremona groups. Note that for such a surface we have a canonical isomorphism 
$H^2(M, \Z) \simeq Pic(M)$ so that
in this case (\ref{H1}) and (\ref{H1-BP}) are the same.
Theorem 1.1 in \cite{BP} describes (\ref{H1-BP}) for a cyclic group $G$ of prime order
and a smooth rational surface $M$ in terms of fixed curves; 
their proof is a case-by-case study based on two-dimensional equivariant minimal model program \cite{Isk}.

\medskip

In this paper we give a uniform proof of Theorem 1.1 in \cite{BP}
and naturally obtain a more general result: 

\begin{theorem}\label{main-thm}
Let $M$ be a smooth projective surface with $H_1(M, \Z) = 0$.
Let $G = \Z/m$ act on $M$ non-trivially. Assume that for every $P \in M$
the stabilizer $\Stab(P)$ is trivial or coincides with $G$ and that the action
has a fixed point.
Then we have a natural isomorphism
\[
H^1(G, H^2(M, \Z)) \simeq \bigoplus_{D \subset M^G} H^1(D, \Z) \otimes \Z/m
\]
where the sum is over all curves $D \subset M$ pointwise fixed by $G$.
\end{theorem}

\begin{remark}
If $G$-action on $M$ is free, then for the smooth quotient $X = M/G$ we have $\chi(X, \OO_X) = \frac1{|G|}\cdot \chi(M, \OO_M)$.
If $\chi(M, \OO_M) = 1$ (for example if $M$ is a rational surface), then the action can not be free, and the condition
about the existence of a fixed point in Theorem \ref{main-thm} is automatically satisfied.
\end{remark}

\begin{example}
Let $M \overset{\pi}{\to} \P^2$ be a double cover ramified in a smooth curve $C$ of degree $2d$.
The cases $d = 1, 2, 3$ correspond to $M$ being a quadric, a del Pezzo surface of degree $2$
and a $K3$ surface respectively. Let $G = \Z/2$ act on $M$ by permuting points in the fibers of $M \to \P^2$.
Let $H \in H^2(M, \Z)$ be the pull-back of $\OO(1)$ from $\P^2$.

Let $r = \rk H^2(M, \Z)$. Since $H^2(\P^2, \Q) = H^2(M, \Q)^G$ the $G = \Z/2$ action on $H^2(M, \Z)$ has $1$ eigenvalue $+1$
and $r-1$ eigenvalues $-1$.

Let $\eps$ denote the nontrivial integral character of $G$. Then $H^1(G, \eps) = \Z/2$.
Consider the short exact sequence of $G$-modules
\[
0 \to \eps^{r-1} \to H^2(M, \Z) \overset{\cdot H}{\to} \Z \to 0
\]
and the corresponding long exact sequence of cohomology groups of $G$:
\[
0 \to H^2(M, \Z)^G \to \Z \to (\Z/2)^{r-1} \to H^1(G, H^2(M, \Z)) \to 0.
\]
One can see that the first map is $\Z \overset{\cdot 2}{\to} \Z$ thus
\[
H^1(G, H^2(M, \Z)) \simeq (\Z/2)^{r-2}. 
\]
Since the only fixed curve of the action is $C$, and it has genus $\frac{(2d-1)(2d-2)}{2}$ 
using Theorem \ref{main-thm} we get
\[
r = \rk H^2(M, \Z) = (2d-1)(2d-2)+2 = 2(2d^2 - 3d + 2)
\]
which for the cases $d=1, 2, 3$ gives $r=2$ (quadric), $r=8$ (del Pezzo surface of degree $2$), $r=22$ ($K3$ surface) respectively.
\end{example}

Our approach to Theorem \ref{main-thm} 
is based on usual Algebraic Topology tools such as the equivariant cohomology group $H^*_G(M,\Z)$,
the Hochschild-Serre spectral sequence and cohomology long exact sequences. These tools are
developed in Sections \ref{sec-equiv}--\ref{sec-surf}.
We finally prove Theorem \ref{main-thm} in Section \ref{sec-surf} by composing two natural isomorphisms
\[
H^1(G, H^2(M, \Z)) \simeq H^{3}_G(M, \Z) \simeq \bigoplus_{D \subset M^G} H^1(D, \Z) \otimes \Z/m.
\]
The intermediate isomorphisms in the chain above
are obtained in Theorem \ref{thm-equiv} (which holds in a more general setting)
and in Theorem \ref{thm-H3}. 
Roughly speaking the reason why $\Z/m$-torsion appears in the equivariant cohomology is that the cohomology
of the infinite lens space $B(\Z/m) = S^\infty / (\Z/m)$ consists of $\Z/m$ in every positive even degree,
see Example \ref{ex-HZp}.

The main technical step in the proof is the long exact sequence of Theorem \ref{thm-seq} which
relates cohomology of the quotient $M/G$ to the $G$-equivariant cohohomology of $M$.

In Sections \ref{sec-chern} and \ref{sec-gysin} we briefly recall the construction of equivariant Chern classes
and equivariant Gysin homomorphism respectively and show that the isomorphism of Theorem \ref{main-thm}
is a particular case of a Gysin homomorphism, see Theorem \ref{thm-gysin} and Remark \ref{rem-gysin}.

\subsection*{Notation}

All varieties are considered over the base field $\C$. The variety $M$ is smooth but not necessarily projective,
and the group $G$ is finite.
All the singular and equivariant singular cohomology and homology groups are considered with integer coefficients,
and the notation $\wt{H}^*(X, \Z)$ stands for the reduced cohomology $\wt{H}^*(X, \Z) = H^*(X, \Z) / H^0(X, \Z)$.

\subsection*{Acknowledgements}

The author is grateful to John Greenlees, Yuri Prokhorov and Constantin Shramov for discussions and encouragement, 
to Tom Sutton for his help with homotopy fibers and cofibers,
to Ieke Moerdijk for e-mail correspondence, and to Hamid Ahmadinezhad and Miles Reid for
the invitation to the Rationality Questions workshop at the University of Bristol
where parts of this work were presented for the first time.

\section{The invariant $H^1(G, H^2(M,\Z))$}
\label{sec-invar}

We say that two smooth projective varieties $M$, $M'$ endowed with a $G$-action are stably $G$-birational if
there exists a $G$-equivariant birational isomorphism between $M \times \P^n$ and $M' \times \P^m$
for some $m, n$, with the $G$-action on $\P^n$ and $\P^m$ assumed to be trivial.

The following lemma is well-known in slightly different context, see \cite[Corollary of Theorem 2.2]{M1}, \cite[2.5, 2.5.1]{BP}:

\begin{lemma}
Let $G$ be a finite group. 
The first group cohomology group $H^1(G, H^2(M, \Z))$ only depends on stable $G$-birational class of 
a smooth projective variety $M$.
\end{lemma}
\begin{proof}
We need to check invariance under $G$-birational isomorphism and multiplication by projective spaces $\P^n$.
The key point in the proof is that for an abelian group $A$ with a trivial $G$-action the first cohomology group is computed as
\[
H^1(G, A) = \Hom(G, A)  
\]
and thus this group is zero as soon as $A$ is torsion-free.

Checking that the cohomology group $H^1(G, H^2(M))$ does not change when $M$ is multiplied by a projective space (with an arbitrary action on it) is straightforward
using the K\"unneth formula:
\[
H^1(G, H^2(M \times \P^n)) = H^1(G, H^2(M) \oplus H^2(\P^n)) = H^1(G, H^2(M)).
\]

Let $M$, $M'$ be smooth projective varieties with $G$-action.
Assume that $M$ and $M'$ are $G$-birational. Then using $G$-equivariant Weak Factorization Theorem
\cite[Remark 2 for Theorem 0.2]{AKMW} $M'$ can be obtained from $M$
using a finite number of blow ups and blow downs with smooth $G$-invariant centers.

Hence it suffices to prove the statement for a blow up $M' = Bl_{Z} (M)$
where $Z$ is a smooth $G$-invariant subvariety of $M$. 
We have:
\[
H^2(M') = H^2(M) \oplus H^0(Z).
\]

Since $Z$ is smooth it is a disjoint union of its irreducible components.
We may restrict ourselves to considering $G$-orbits of components in which case
\[
H^0(Z) = \Hom(\Z[\pi_0(Z)], \Z) = \Hom(\Z[G/H], \Z) = \Hom_H(\Z[G], \Z)
\]
where $H \subset G$ is a stabilizer of an irreducible component. By Shapiro's Lemma \cite[6.3.2]{Wei} 
\[
H^1(G, H^0(Z)) = H^1(H, \Z) = 0.
\]
and finally 
\[
H^1(G, H^2(M')) = H^1(G, H^2(M)).
\]
\end{proof}

We call the $G$-action on $M$ stably linearizable
if 
$M \times \P^n$ is $G$-birational to $\P^N$ for some $G$-action on $\P^N$.

\begin{corollary}\cite[2.5.2]{BP}
If $G$-action on $M$ is stably linearizable then
\[
H^1(G, H^2(M)) = 0. 
\]
\end{corollary}
\begin{proof}
The obstruction group for any $G$-action on $\P^N$ vanishes:
\[
H^1(G, H^2(\P^N)) = 0
\]
since $H^2(\P^N) = \Z$ is a trivial $G$-module as the ample generator must be preserved by the $G$-action.
\end{proof}

\section{Equivariant cohomology}
\label{sec-equiv}

Let $M$ be a smooth variety with a finite group $G$ acting on it.
The equivariant cohomology groups have been defined by Borel \cite[Chapter IV]{B} as
\[
H^*_G(M, \Z) = H^*((M \times EG) / G, \Z). 
\]
Here $EG$ is a contractible space with a free action of $G$.
In modern terms equivariant cohomology groups of the topological realization
of the Deligne-Mumford stack $M/G$,
see \cite{Ieke} for an accessible introduction.

We note that 
\[
H^*_G(\pt, \Z) = H^*(BG, \Z) \simeq H^*(G, \Z) 
\]
where $BG = EG / G$ and 
the right most term is the group cohomology of $G$ with coefficients in a trivial module $\Z$.

\begin{example}\label{ex-HZp}
Let $G=\Z/m$ be a cyclic group. Then $\Z/m$ acts freely on the infinite-dimensional sphere $S^\infty$ multiplying
all the coordinates by the primitive $m$-th root of unity.
Thus $B(\Z/m)$ is an infinite lens space $L_\infty(m) = S^\infty / G$ and it is well-known that:
\[
H^j(BG, \Z) =  \left\{\begin{array}{ll}
\Z  , & j = 0\\
0  , & j \; odd\\
\Z/m  , & j \; even, \; j > 0\\
\end{array}\right.
\]
see \cite[Example 6.2.3]{Wei} for the algebraic approach or \cite[(18.6) and Exercise 18.8(a)]{BT} for
the topological argument.
\end{example}

\begin{remark}
If the $G$-action on $M$ is free, then 
\[
H^*_G(M, \Z) = H^*((M \times EG)/G, \Z) \simeq H^*(M/G, \Z). 
\]
This is because in this case the projection map 
$(M \times EG)/G \to M/G$ is a fibration with contractible fiber $EG$,
hence a homotopy equivalence.

In the other extreme case when the $G$-action on $M$ is trivial, then
\[
H^*_G(M, \Z) = H^*((M \times EG)/G, \Z) =  H^*(M \times BG, \Z).
\]
In particular if cohomology of $M$ are torsion free, then
\[
H^*_G(M, \Z) = H^*(M, \Z) \otimes H^*(BG, \Z) 
\]
by the K\"unneth formula.
\end{remark}

The following Theorem explains the relationship between $H^*_G(M, \Z)$ and
$H^*(M/G, \Z)$ in general. 
We do not assume $M$ to be compact and in this setting it is natural to consider 
equivariant cohomology groups with compact supports 
\[
H^j_{c,G}(M, \Z) := H^j_{c}((M \times EG) / G, \Z). 
\]

\begin{theorem}\label{thm-seq}
Let finite group $G$ act on a smooth variety $M$, and let $X = M/G$.
Assume that for every $P \in M$ the stabilizer of $P$ is trivial or coincides with $G$. Let $M^G$ denote the (possibly disconnected) fixed locus of $G$.
Then there is a long exact sequence of cohomology groups with compact supports:
\be{eq-seq-main}
\dots \to K^{j-1} \to H^j_c(X, \Z) \to H^j_{c,G}(M, \Z) \to K^j \to \dots
\ee
with $K^j = H^j_{c}(M^G \times BG,\Z) / H^j_c(M^G,\Z)$.
In particular, if $H^*_c(M^G,\Z)$ is torsion-free, then $K^* = H^*_c(M^G,\Z) \otimes \wt{H}^*(BG, \Z)$ by the K\"unneth formula.

The long exact sequence (\ref{eq-seq-main}) is covariantly functorial for equivariant open embeddings $V \subset M$.
\end{theorem}

\begin{remark}
The assumption on the stabilizers of the group action in Theorem \ref{thm-seq} implies that the fixed locus $M^G$ is a disjoint union of smooth subvarieties of $M$.
Indeed $M^G$ coincides with the fixed locus $M^g$ of every nontrivial element $g \in G$, and the $M^g$'s are always smooth.
\end{remark}

\begin{remark}
The assumption on the stabilizers in Theorem \ref{thm-seq} are automatically satisfied if $G = \Z/p$ for a prime $p$.
\end{remark}

\begin{proof}[Proof of Theorem \ref{thm-seq}]
Let $U = M - M^G$, so that $G$ acts freely on $U$ and trivially on $M^G$ by assumption.
We compare two long exact sequences of compactly supported cohomology groups \cite[III.7.6, III.7.8]{Iv}:
\[\xymatrix{
\dots \ar[r] & H^{i-1}_{c,G}(M^G) \ar[r] & H^{i}_{c,G}(U) \ar[r] & H^i_{c,G}(M) \ar[r] & \dots \\
\dots \ar[r] &H^{i-1}_c(M^G) \ar[r] \ar[u]& H^{i}_c(U/G) \ar[r] \ar[u]_{\simeq} & H^i_c(M/G)  \ar[r] \ar[u] & \dots  \\
}\]

The vertical arrows in the diagram are given by pull-backs: pull-backs are well-defined
because classifying spaces $BG$ can be modeled by CW-complexes with finitely many cells in every dimension,
so the proper pull-back of \cite[III.7.8]{Iv} applies.

Since the $G$-action on $U$ is free, the middle vertical arrow is an isomorphism \cite[IV.3.4]{B}.
Consider the cone of the map of exact sequences above to get an exact sequence:
\[
\dots \to H^{i-1}_{c,G}(M^G) \oplus H^i_{c}(U/G) \to H^i_{c,G}(U) \oplus H^i_{c}(M/G) \to H^i_{c,G}(M) \oplus H^i_c(M^G) \to \dots
\]
Splitting off the terms $H^{i}_{c}(U/G)  \simeq H^{i}_{c,G}(U)$ we obtain an exact sequence
\[
\dots \to H^{i-1}_{c,G}(M^G) \to H^i_{c}(M/G) \to H^i_{c,G}(M) \oplus H^i_c(M^G) \to H^{i}_{c,G}(M^G) \to \dots
\]
Since $G$ acts trivially on the fixed locus $M^G$, the map 
\[
H^i_c(M^G) \to H^{i}_{c,G}(M^G) = H^i_{c}(M^G \times BG)
\]
is a split injection. This yields the desired long exact sequence (\ref{eq-seq-main}).

For functoriality of (\ref{eq-seq-main}) note that all the maps here come from the cohomology long exact sequence \cite[III.7.6]{Iv}.
When checking functoriality one uses that the latter cohomology exact sequence 
is covariantly functorial for proper pull-backs  (by \cite[III.7.8]{Iv}) and for open embeddings (by a slight generalization
of \cite[III.7.7]{Iv}).
\end{proof}

\begin{remark}
Note that the $K^j$'s are torsion abelian groups.
Thus if we consider cohomology with rational coefficients, we recover the usual
isomorphism $H^*_{c,G}(M, \Q) = H^*_c(M/G, \Q)$.
\end{remark}

\begin{example}\label{ex-V}
Let $M = V = \C^n$, and let $G = \Z/m$ action on $V$ be generated by
\[
(z_1, \dots, z_n) \mapsto (\zeta^{a_1} z_1, \dots, \zeta^{a_n} z_n)
\]
where $\zeta$ is the $m$-th root of unity.
Assume that all the weights $a_i \in \Z/m$ are non-zero.
In this case the fixed locus is $M^G = \{0\} \subset V$ so that 
\[
K^j =  \wt{H}^j(BG) = \left\{\begin{array}{ll}
\Z/m  , & j > 0 \; \text{even}\\
0  , & \text{otherwise} \\
\end{array}\right.
\]
Furthermore, since $V$ is equivariantly homemorphic to the cone over $S^{2n-1}$,
$X = V/G$ can be identified with a cone over the lens space $S^{2n-1} / G = L_{2n-1}(m)$.
Using the long cohomology sequence with compact supports we compute cohomology of the cone:
\[
H^*_{c}(X) = \wt{H}^{*-1}(L_{2n-1}(m)).
\]
The lens space $L_{2n-1}(m)$ is a compact $(2n-1)$-dimensional real manifold and its cohomology is:
\[
H^j(L_{2n-1}(m))  = \left\{\begin{array}{ll}
\Z/m  , & 2 \le j \le 2n-2 \; \text{even}\\
\Z,     & j = 0, 2n-1 \\
0  , & \text{otherwise} \\
\end{array}\right.
\]
see e.g. \cite[(18.6)]{BT}, so that 
\[
H^j_c(X)  = \left\{\begin{array}{ll}
\Z/m  , & 3 \le j \le 2n-1 \; \text{odd}\\
\Z,     & j = 2n \\
0  , & \text{otherwise} \\
\end{array}\right.
\]
Finally, we can compute $H^*_{c,G}(V)$ using the fact that $(V \times EG) / G$ is a complex vector bundle of rank $n$ over $BG$:
\[
H^j_{c,G}(V) = H^{j-2n}(BG) = 
\left\{\begin{array}{ll}
\Z,     & j = 2n \\
\Z/m  , & j > 2n \; even\\
0  , & \text{otherwise} \\
\end{array}\right.
\]

These computations are in accordance with the exact sequence of Theorem \ref{thm-seq}.
In particular we see that the boundary maps
\be{eq-delta}
\partial^{2i}: \Z/m = K^{2i} \to H^{2i+1}_c(V/G), \;\; 0 < 2i < 2n
\ee
are isomorphisms sending generators of
$\Z/m$ to the cohomology classes Poincare dual to the homology classes of the lens space space 
$S^{2n-1} / G \subset V/G$.
\end{example}

In the next section we make the boundary map in the long exact sequence of Theorem \ref{thm-seq} explicit
for isolated fixed points, thus generalizing (\ref{eq-delta}).

\section{Cohomology of cyclic quotient singularities}

Let $X$ be a $n$-dimensional variety, and let $P \in X$ be an isolated quotient singularity of order $m$.
Consider the lens space $L_{2n-1}(m)$ bounding a small contractible neighbourhood of $P$.
Recall that homology groups of the lens space are given by:
\[
H_k(L_{2n-1}(m))  = \left\{\begin{array}{ll}
\Z/m  , & 1 \le j \le 2n-3 \; \text{odd}\\
\Z,     & k = 0, 2n-1 \\
0  , & \text{otherwise} \\
\end{array}\right.
\]

The homology of $L_{2n-1}(m)$ may contribute to the odd degree cohomology of $X$.

\begin{definition}
For every $1 \le i \le n-1$ we define the map 
\[
\partial^{2i}_P: \Z/m \to H^{2i+1}_c(X, \Z)
\]
as a composition
\[
\Z/m = H_{2n-2i-1}(L_{2n-1}(m)) \to 
H_{2n-2i-1}(X^\circ) \simeq H^{2i+1}_c(X^\circ) = H^{2i+1}_c(X)
\]
of push-forward on homology, Poincare duality for nonsingular locus $X^\circ$ of $X$
and a natural isomorphism coming from long exact cohomology sequence with compact supports
(all coefficients in the cohomology groups $\Z$, as usual).
Using a slight abuse of notation we denote the postcomposition of $\partial^{2i}_P$ with the map 
\[
H^{2i+1}_c(X, \Z) \to H^{2i+1}(X, \Z) 
\]
by the same symbol $\partial^{2i}_P$:
\[
\partial^{2i}_P: \Z/m \to H^{2i+1}(X, \Z).
\]
\end{definition}

\begin{lemma}\label{lem-delta}
 Let $M$ be a smooth $n$-dimensional variety with $G = \Z/m$ action. Let $P \in M$ be an isolated fixed point of $G$.
 For every $1 \le i \le n-1$ let $\delta^{2i}_P$ be the restriction of the map $K^{2i} \to H^{2i+1}_c(M/G)$ of Theorem \ref{thm-seq}
 to its component $\Z/m = {H}^{2i}(BG, \Z) \subset K^{2i}$ corresponding
 to $P \in M^G$. Then the map $\delta^{2i}_P$ coincides with $\partial^{2i}_P$ defined above.
\end{lemma}
\begin{proof}
Let $V$ be a small $G$-equivariant open neighbourhood of $P \in M$ homeomorphic to $\C^n$.
Since the exact sequence of Theorem \ref{thm-seq} is covariantly functorial with respect to open inclusions,
we get a commutative diagram
\[\xymatrix{
K^{2i} \ar[r] & H^{2i+1}_c(M/G) \\
\Z/m \ar[u] \ar[r] & H^{2i+1}_c(V/G) \ar[u] \\
}\]
Thus we are reduced to the case $M = V$ considered in Example \ref{ex-V}.
In this case the claim follows from the definitions.
\end{proof}

\begin{lemma}\label{lem-delta-surj}
Let $X$ be a projective $n$-dimensional variety
with $H_1(X, \Z) = 0$ 
with isolated cyclic quotient singularities $P_1, \dots, P_r \in X$ of orders $m_1, \dots, m_r$
and no other singularities.
Then the map
\[
\partial = \sum \partial^{2n-2}_{P_i}: \bigoplus_{i=1}^r \Z/m_i \to H^{2n-1}(X,\Z)
\]
is surjective.
\end{lemma}
\begin{proof}
We consider the cohomology exact sequence for the pair $X^\circ \subset X$:
\[
H^{2n-1}(X^\circ / X,\Z) \overset{\partial'}{\to} H^{2n-1}(X,\Z) \to H^{2n-1}(X^\circ, \Z).
\]
The Thom space $X^\circ / X$ is homotopy equivalent to a wedge of suspensions of the
lens spaces $L_{2n-1}(m_i)$, so that
\[
H^{2n-1}(X^\circ / X, \Z) = \bigoplus_{i=1}^r \Z/m_i
\]
and under this identification the map $\partial'$ corresponds to $\partial$.

Since $X$ is a $\Q$-homology manifold, $H_1(X,\Z) = 0$ implies that $H^{2n-1}(X,\Z)$
is torsion. Thus to show that $\partial$ is surjective 
it suffices to check that $H^{2n-1}(X^\circ, \Z)$ is torsion-free.

We use Poincare duality isomorphism of cohomology and Borel-Moore homology \cite[Theorem IX.4.7]{Iv}:
\[
H^{2n-1}(X^\circ, \Z) = H_1^{BM}(X^\circ, \Z)
\]
and the long exact Borel-Moore homology sequence \cite[IX.2.1]{Iv}:
\[
0 = H_1(X, \Z) \to H_1^{BM}(X^\circ, \Z) \to H_0(\{P_1, \dots, P_r\}, \Z) = \Z^r
\]
which implies that the group in the middle is torsion-free.
\end{proof}

\section{The isomorphism $H^1(G, H^2(M, \Z)) \simeq H^3_G(M, \Z)$}

\begin{theorem}\label{thm-equiv}
Let $M$ be a smooth projective variety with a nontrivial $G=\Z/m$ action.
Assume that $H^1(M, \Z) = 0$, $H^3(M, \Z)^G = 0$ 
and that the action has at least one fixed point.
Then we have a natural isomorphism
\[
H^1(G, H^2(M, \Z)) \simeq H^3_G(M, \Z).
\]
\end{theorem}
\begin{proof}
We use the Hochschild-Serre spectral sequence 
\[
E_2^{p,q} = H^p(G, H^q(M, \Z)) \implies H^{p+q}_G(M,\Z)
\]
which can be constructed as the Leray-Serre spectral sequence of the
fiber bundle $(M \times EG)/G \to BG$ with fiber $M$.

All the groups on the $p+q=3$ diagonal of second term of 
vanish except for $E_2^{1,2} = H^1(G, H^2(M, \Z))$ 
(see Example \ref{ex-HZp} for group cohomology of $\Z/m$ with constant coefficients).

The only non-trivial differential touching the latter group is
\[
d_3: H^1(G, H^2(M, \Z)) \to H^4(G, H^0(M, \Z)) 
\]
with the spectral sequence looking as follows:

\[
\xymatrix{
_3 \ar@{.}[r] & 0 \ar@{.}[d] \\
_2 \ar@{--}[u] \ar@{.}[r] & 0 \ar@{.}[r] \ar@{.}[d] & E_2^{1,2} \ar@{->}[ddrrr]^{d_3} \ar@{.}[d] \\
_1 \ar@{--}[u] \ar@{.}[r] \ar@{.}[d] & 0 \ar@{.}[d] \ar@{.}[r] & 0 \ar@{.}[r] \ar@{.}[d] & 0 \ar@{.}[r] \ar@{.}[d] & 0 \ar@{.}[d] \\
_0 \ar@{--}[u] \ar@{.}[r] & \Z \ar@{.}[r] & 0 \ar@{.}[r] & \Z/m \ar@{.}[r] & 0 \ar@{.}[r] & E_2^{4,0} \\
^q/_p \ar@{--}[u] \ar@{--}[r] & _0 \ar@{.}[u] \ar@{--}[r] & _1 \ar@{.}[u] \ar@{--}[r] & _2 \ar@{.}[u] \ar@{--}[r] & _3 \ar@{--}[r] \ar@{.}[u] & _4 \ar@{.}[u]
}
\]

We now show that the differential $d_3$ is zero. 
Indeed since it is also the only non-trivial differential touching the target group
\[
E_2^{4,0} = H^4(G, H^0(M, \Z)) = H^4(G, \Z) = \Z/m
\]
we have
\[
Im(d_3) = Ker(\phi: \Z/m \to H^4_G(M, \Z)),
\]
where $\phi$ is the boundary morphism $E_2^{4,0} \to E_\infty^{4,0}$.

The boundary morphism $\phi$ is the same thing
as the pull-back morphism $H^4_G(pt, \Z) \to H^4_G(M, \Z)$.
The existence of a fixed point of $G$ gives a splitting of $\phi$, so that $\phi$ is
injective, and 
\[
Im(d_3) = Ker(\phi) = 0.
\]

It follows now from the considerations above that the boundary map
\[
E^{1,2}_2 = H^1(G, H^2(M, \Z)) \to H^3_G(M, \Z) = E^{1,2}_\infty
\]
is an isomorphism.

\end{proof}

\section{Cohomology of surfaces}
\label{sec-surf}

\begin{theorem}\label{thm-H3}
Let $M$ be a smooth projective surface with $H_1(M,\Z) = 0$. Let $M$ have a $G = \Z/m$ action such that for every $P \in M$ 
the stabilizer $\Stab(P)$ is trivial or is equal to $G$ and that the action has a fixed point.
There is a canonical isomorphism
\[
H^3_G(M, \Z) \simeq \bigoplus_{D \subset M^G} H^1(D, \Z) \otimes \Z/m
\]
where is the sum is taken over curves in the fixed locus $M^G$, i.e. over all point-wise fixed curves in $M$. 
\end{theorem}
\begin{proof}
We use the long exact sequence of Theorem \ref{thm-seq}:
\[
K^2 \to H^3(M/G, \Z) \to H^3_G(M, \Z) \to K^3 \to H^4(M/G,\Z). 
\]
The group $H^4(M/G,\Z)$ is a free abelian group generated by connected components of $M/G$.
Since $K^3$ is a torsion group the right most map in the exact sequence is zero.

Since the fixed locus $M^G$ of the action is a disjoint union of points and curves the cohomology of $M^G$ is torsion free
and groups $K^2$, $K^3$ are computed as
\[\bal
K^2 &= H^0(M^G, \Z) \otimes H^2(BG, \Z) = (\Z/m)^{\pi_0(Z)} \\
K^3 &= H^1(M^G, \Z) \otimes H^2(BG, \Z) = \bigoplus_{D \subset M^G} H^1(D, \Z) \otimes \Z/m.
\eal\]
Thus the long exact sequence above rewrites as
\[
(\Z/m)^{\pi_0(M^G)} \to H^3(M/G, \Z) \to H^3_G(M, \Z) \to \bigoplus_{D \subset M^G} H^1(D,\Z) \otimes \Z/m \to 0.  
\]
and we need to check that the first map on the left is surjective.

Since the $G$-action on $M$ has a fixed point $P$, it follows from Lemma \ref{lem-H1}
that $H_1(X, \Z) = 0$.
Let $P_1, \dots, P_r \in M$ denote the isolated fixed points of the $G$-action.
Lemma \ref{lem-delta-surj} tells that the boundary map
\[
\sum_{i=1}^r \partial^{2}_{P_i}: (\Z/m)^r \to H^3(M/G, \Z) 
\]
is surjective, hence by Lemma \ref{lem-delta} so is the map
\[
\delta^2: (\Z/m)^{\pi_0(M^G)} \to H^3(M/G, \Z).
\]
\end{proof}

\begin{proof}[Proof of Theorem \ref{main-thm}]
Using Poincare duality we can write:
\[\bal
H^0(M, \Z) &\simeq \Z \\
H^1(M, \Z) &\simeq \Z^{b_1} \\
H^2(M, \Z) &\simeq \Z^{b_2} \oplus T\\
H^3(M, \Z) &\simeq \Z^{b_1} \oplus T\\
H^4(M, \Z) &\simeq \Z
\eal\]
where $T \simeq H_1(M, \Z)_{tors}$ and $b_i = \rk \, H_i(M, \Z)$,
thus
\[
H_1(M, \Z) = 0 \implies H^1(M, \Z) = H^3(M, \Z) = 0. 
\]

All conditions of Theorem \ref{thm-equiv} are satisfied and we have
\[
H^1(G, H^2(M, \Z)) \simeq H^3_G(M, \Z).
\]

We finish the proof using Theorem \ref{thm-H3}.
\end{proof}

\begin{lemma}\label{lem-H1}
Let $G$ act on a variety $M$ and assume that $G$ has a fixed point. Let $X = M/G$. Then the maps
\[
\pi_1(M) \to \pi_1(X) 
\]
\[
H_1(M, \Z) \to H_1(X, \Z) 
\]
are surjective.
\end{lemma}
\begin{proof}
Let $P \in M$ be a fixed point.
By \cite[Lemma 1]{Arm} the quotient map $p: M \to M/G$ admits path lifting property.
Thus every loop in $X$ based in $p(P)$ lifts to a loop in $M$ based in $P$, and
the map of fundamental groups is surjective.
It follows that the map on first homology groups is surjective as well.
\end{proof}

\section{$G$-equivariant Chern classes}\label{sec-chern}

For a $G$-equivariant vector bundle $E$ on a variety $M$ with a $G$-action one may
define equivariant Chern classes $c_i^G(E) \in H^{2i}_G(M, \Z)$ by assigning to $G$-equivariant
bundles vector bundles on $(M \times EG)/G$, that is via the composition
\[
c_i^G: K_0^G(M) \to K_0((M \times EG)/G) \overset{c_i}{\to} H^{2i}_G(M, \Z).
\]

As a particular case every representation $V$ of $G$ determines a $G$-equivariant
vector bundle on a point, so in particular we have Chern classes $c_i(V) \in H^{2i}(BG, \Z)$ studied in detail in \cite{At}.

\begin{example}\label{ex-BG-mult}
Let $G = \Z/m$, and let $\rho$ be the character $\ol{k} \mapsto e^{{2\pi i \cdot k}/m}$.
Let $x = c_1(\rho) \in H^2(BG, \Z) = \Z/m$. Then $c_1(\rho^{\otimes j}) = jx$.
In fact there is an isomorphism of rings \cite[\S8]{At}
\[
H^*(BG, \Z) = \Z[x] / mx. 
\]
\end{example}

The equivariant Chern classes satisfy all the usual properties of Chern classes
such as functoriality, formulas for direct sums and tensor products etc.

We need some results on equivariant Chern classes in the case of trivial $G$-action.

\begin{lemma}\label{lem-cn}
Let $Z$ be a variety endowed with trivial $G$-action, let $E$ be a vector bundle on $Z$ of rank $n$, and
let $\rho: G \to \C^*$ be a character of $G$. Then the top Chern class of the equivariant 
$G$-bundle $E \otimes \rho$ is given by
\[
c_n^G(E \otimes \rho) = \sum_{i=0}^n c_{i}(E) c_1(\rho)^{n-i} \in H^{2n}_G(Z, \Z).
\]
\end{lemma}
\begin{proof}
Let $L_\rho$ be the line bundle on $BG$ corresponding to $\rho$.
Then $E \otimes \rho$ considered as a vector bundle on $(Z \times EG) / G = Z \times BG$ 
is the bundle $p_1^*(E) \otimes p_2^*(L_{\rho})$.
We apply the formula for the total Chern class of a tensor product \cite[21.10]{BT} to obtain
\[
\sum_{i=0}^n c^G_i(E \otimes \rho) =  \sum_{i=0}^n c_i(E) (1+c_1(\rho))^{n-i}. 
\]
Considering the top degree terms yields the desired formula for the $n$-th Chern class.
\end{proof}

Assume now $G = \Z/m$, and let $Z$ be a variety endowed with trivial $G$-action.
Then every $G$-equivariant vector bundle $E$ on $Z$ decomposes into its isotypical components
\[
E = \bigoplus_{j=0}^{n-1} E_j \otimes \rho^{\otimes j}
\]
where $E_j$ is a vector bundles on $Z$, and $\rho$ is the character of $\Z/m$
sending the generator to $e^{{2\pi i}/m}$.

As in Example \ref{ex-BG-mult} let $x = c_1(\rho) \in H^2(BG, \Z)$ be the canonical generator
and using a slight abuse of notation we also denote by $x \in H^2_G(M, \Z)$ the pull-back 
of $x \in H^2(BG, \Z)$.

\begin{lemma}\label{lem-cn-sum}
In the above setup let $E$ be a $G$-equivariant vector bundle on $Z$ of rank $n$.
Then the top equivariant Chern class of $E$ is a homogeneous polynomial in $x$ of degree $2n$ with coefficients in $H^*(Z, \Z)$ of the form
\[
c_n^G(E) = \left( \prod_{j=0}^{n-1} j^{\rk(E_j)}\right) x^n + \; \text{terms of lower degree in $x$} 
\]
\end{lemma}
\begin{proof}
Let $n_j = \rk(E_j)$.
We compute using Lemma \ref{lem-cn} and the fact that $c_1(\rho^{\otimes j}) = jx$:
\[\bal
c_n^G(E) &= \prod_j c_{n_j}^G (E_j \otimes \rho^{\otimes j}) =\\
&= \prod_j \left( \sum_{i=0}^{n_j} c_i(E_j)^G c_{1}(\rho^{\otimes j})^{n_j-i} \right) =\\
&= \prod_j \left( \sum_{i=0}^{n_j} c_i(E_j)^G (jx)^{n_j-i} \right) =\\
&= \prod_j \left( j^{n_j} x^{n_j} + \;\text{lower degree terms} \right) =\\
&= \left( \prod_{j} j^{n_j}\right) x^n + \;\text{lower degree terms}.
\eal\]
\end{proof}

\section{Equivariant Gysin homomorphism}\label{sec-gysin}

For a $G$-equivariant embedding of smooth algebraic varieties $Z \subset M$ we may
define the Gysin homomorphism
\[
i_*^G: H^*_G(Z, \Z) \to H^{*+2e}_G(M, \Z). 
\]
Here $e$ is the codimension of $Z$ in $M$. 
The equivariant Gysin homomorphism is constructed using the ordinary Gysin homomorphism
applied to maps $(Z \times (EG)_N)/G \to (M \times (EG)_N)/G$,
where $(EG)_N$ is an $N$-universal $G$-space, that is a space on which $G$ acts freely with vanishing 
cohomology groups in degrees $0 \le i \le N$. As explained in \cite[IV.3.1]{B}, equivariant cohomology groups
may be computed using $N$-universal $G$-spaces when $N$ is large enough.

Since $(EG)_N$ may be chosen to be smooth manifolds, this allows the construction
of the Gysin homomorphism which then can be shown to be independent of all the choices made
and to satisfy the usual properties such as functoriality and relation to Chern classes.

In the case when the $G$-action on $Z$ is trivial, we have $H^*_G(Z, \Z) = H^*(Z \times BG, \Z)$
so that for every $\gamma \in H^d(BG, \Z)$ there is a homomorphism
\[
H^j(Z, \Z) \overset{\boxtimes \gamma}{\to} H^{j+d}_G(Z, \Z) \overset{i_*^G}{\to} H^{j+d+2e}_G(M, \Z) 
\]
which me may refer to as the $\gamma$-component of map $i_*^G$.

In particular, when $G = \Z/m$, we consider $x^i$-components, where $x \in H^2_G(Z, \Z)$ is defined in the discussion
preceding Lemma \ref{lem-cn-sum}.

\begin{theorem}\label{thm-gysin}
Let $M$ be a smooth variety with a non-trivial $G = \Z/m$-action. Assume that for every $P \in M$ the 
stabilizer $\Stab(P)$ is trivial or coincides with $G$.
Let $Z$ be a union of connected components of $M^G$ of codimension $e$ in $M$.
For all $i,j \ge 0$, the $x^i$-components of the Gysin homomorphism induce injective maps
\[
H^j(Z, \Z) \otimes \Z/m \to H^{j+2i+2e}_G(M, \Z) \otimes \Z/m. 
\]
\end{theorem}
\begin{proof}
%We recall that $H^*(BG, \Z) = \Z[x]/(mx)$ \cite[]{At} so that we have an embedding of cohomology rings
%\[
%H^*(M^G)[x] / (mx) = H^*(M^G) \otimes H^*(BG) \subset H^*(M^G \times BG) = H^*_G(M^G).
%\]

We first assume that $Z$ consists of only one component, and then explain the general case.
Let $i: Z \to M$ denote the embedding. 
%Let $e = \codim_M(Z)$ and
Let $\NN$ be the normal bundle of $Z$.

From the standard properties of the Gysin homomorphism it follows that 
the composition $i^*_G i_*^G: H^*_G(Z, \Z) \to H^*_G(Z, \Z)$
is a multiplication map by the element
\[
\eps := i^*_G i_*^G(1) = c_e^G(\NN) \in H^*_G(Z, \Z),
\]
the top equivariant Chern class of the normal bundle.

We use Lemma \ref{lem-cn-sum} to write an expression for $\eps$.
The top coefficient of $\eps$ considered as a polynomial in $x \in H^2(BG, \Z)$ is
$a = \left( \prod_{j} j^{\rk(\NN_j)}\right) \in \Z/m$.

For every $g \in G$ the normal bundle $N_{M^g/M}$ has no component with $g$ acting trivially.
Since for every $g \ne 0$ we have $M^g = M^G$, this implies that $\NN_j = 0$ unless $\gcd(j, n) = 1$.
This means that the coefficient $a \in \Z/m$ is invertible.

We summarize our results in the commutative diagram:
\[\xymatrix{
& H^{j+2i+2e}_G(M) \otimes \Z/m \ar[dr]^{i^*_G} & \\
H^{j+2i}_G(Z) \otimes \Z/m  \ar[ur]^{i_*^G} \ar[rr]^{\cdot \eps} & & H^{j+2i+2e}_G(Z) \otimes \Z/m  \ar[d]^{\text{coef of} \; x^{i+e}} \\
H^{j}(Z) \otimes \Z/m \ar[u]^{\boxtimes x^i} \ar[rr]^{\cdot a} & & H^{j}(Z) \otimes \Z/m \\
}\]

Since $a$ is invertible, the component $H^j(Z, \Z) \otimes \Z/m \to H^{j+2i+2e}(M, \Z) \otimes \Z/m$ is injective.

Now for the general case, that is when $Z$ consists of several components $Z_1, \dots, Z_s$ of codimension $e$ 
the same argument goes through and we obtain a homomorphism
\[
\bigoplus_{k=1}^s H^j(Z_k, \Z) \otimes \Z/m \to \bigoplus_{k=1}^s H^j(Z_k, \Z) \otimes \Z/m.
\]
Since by the assumption on the fixed points different components of $Z$ do not intersect 
this homomorphism is represented by a diagonal matrix. The diagonal elements were seen above to be units of $\Z/m$,
and this shows that the homomorphism $\oplus_{k=1}^s H^j(Z_k, \Z) \otimes \Z/m \to H^{j+2i+2e}_G(M,\Z) \otimes \Z/m$
is injective.
\end{proof}

\begin{corollary}
For a smooth variety $M$ with a $G = \Z/m$-action satisfying the assumption on stabilizers from Theorem \ref{thm-gysin}
the components of the Gysin homomorphism give an embedding
\[
\bigoplus_{D \subset M^G} H^1(D, \Z) \otimes \Z/m \subset H^3_G(M, \Z) \otimes \Z/m
\]
where the sum is over all divisors $D \subset M$ pointwise fixed by $G$.
\end{corollary}

\begin{remark}\label{rem-gysin}
The proof of Theorem \ref{thm-gysin} shows that in the case when assumptions of Theorem 1.1 are satisfied, 
that is when $M$ is smooth projective and $H_1(M, \Z) = 0$, then
the isomorphism constructed in the proof of Theorems \ref{main-thm}, \ref{thm-H3}, 
which was a restriction map of equivariant cohomology,
is up to an invertible scalar multiple on each component $M^G$,
the inverse of the Gysin homomorphism $H^1(M^G, \Z) \otimes \Z/m \to H^3_G(M, \Z)$
(in this case the target cohomology group is $m$-torsion by Theorem \ref{thm-equiv}, so we don't need to tensor it with $\Z/m$).
\end{remark}

\begin{example}
Let $M = \P^1$, and let the generator of $G=\Z/m$ act on $M$ by $[x:y] \mapsto [x:\zeta y]$, where $\zeta = e^{2\pi i/m}$ is a primitive $m$-th root of unity.

In this case $H^*_G(\P^1)$ as well as its ring structure can be computed explicitly since the homotopy quotient $(\P^1 \times EG) / G$ is a $\P^1$-bundle
over $BG$. More precisely, let $L$ be a line bundle on $BG$ associated to the character $\rho: z \mapsto \zeta z$, then
\[
(\P^1 \times EG) / G = \P_{BG}(\OO \oplus L).
\]

The cohomology algebra of a projective bundle depends on the Chern classes of the vector bundle \cite[(20.6)]{BT}.
By our conventions $c_1(L) = x$, the generator of $H^2(BG, \Z)$ so that we have $c_1(\OO \oplus L) = x$ and $c_2(\OO \oplus L) = 0$.
We find that the equivariant cohomology of $\P^1$ as an algebra over $H^*(BG, \Z) = \Z[x]/(mx)$ is given by
\[
H^*_G(\P^1, \Z) = H^*(\P_{BG}(\OO \oplus L),\Z) = \Z[x,h] / (mx, h^2+xh), 
\]
so that
\[
H^j_{G}(\P^1) = 
\left\{\begin{array}{ll}
\Z \cdot 1,     & j = 0 \\
\Z/m \cdot x \oplus \Z \cdot h, & j = 2 \\
\Z/m \cdot x^k \oplus \Z/m \cdot x^{k-1}h, & j = 2k > 2 \\
0, & j \; \text{odd} \\
\end{array}\right.
\]

The two fixed points $0 = [1:0]$ and $\infty = [0:1]$ correspond to the two disjoint sections of the $\P^1$-bundle:
\[\bal
i_0: BG &= \P(\OO) \subset \P(\OO\oplus L) \\
i_\infty: BG &= \P(L) \subset \P(\OO\oplus L), \\
\eal\]
and one can compute that the Gysin homomorphisms 
\[
i_{0, *}, i_{\infty,*}: H^*(BG, \Z) \to H^*_G(\P^1)
%i_{0,*}(1) &= x \in H^2_G(\P^1) = \Z \oplus \Z/m\\
%i_{\infty,*}(1) &= x+h \in H^2_G(\P^1) = \Z \oplus \Z/m\\
\]
are multiplication maps by $h+x$ and $h$ respectively.
The sum of the two homomorphisms
\[
i_{0,*} + i_{1,*}: H^*(BG, \Z)^{\oplus 2} \to H^*_G(\P^1, \Z) 
\]
as well as its reduction modulo $m$ is injective, in
accordance with Theorem \ref{thm-gysin}.

This example demonstrates that as opposed to the case of Theorem \ref{thm-H3} there is no inclusion $H^i(Z, \Z) \otimes \Z/m \to H^{i+2e}_G(M, \Z)$ 
($Z$ is the union components of fixed locus of codimension $e$ in $M$)
in general: the action of $G$ on $\P^1$ has two fixed points, but the $m$-torsion in $H^2_G(\P^1, \Z)$ is only $\Z/m$.
\end{example}

\medskip
\medskip

\address{
{\bf Evgeny Shinder}\\
School of Mathematics and Statistics \\
University of Sheffield \\
The Hicks Building \\
Hounsfield Road \\
Sheffield S3 7RH\\
e-mail: {\tt eugene.shinder@gmail.com}
%\email{evgenyshinder2011@u.northwestern.edu}
}

\end{document}